\documentclass[11pt]{amsart}
\usepackage{amsmath,amscd,amssymb}
\usepackage{hyperref}

\newtheorem{theorem}{Theorem}[section]

\newtheorem{proposition}[theorem]{Proposition}

\newtheorem{lemma}[theorem]{Lemma}
\theoremstyle{remark}
\newtheorem{remark}[theorem]{Remark}

\newtheorem{example}[theorem]{Example}

\newcommand\da{\dasharrow}

\newcommand\be{\begin{equation}\label}
\newcommand\ee{\end{equation}}

\newcommand{\V}{\mathcal{V}}
\renewcommand{\O}{\mathcal{O}}

\newcommand{\Co}{\mathcal{C}}

\newcommand{\R}{\mathbb{R}}
\newcommand{\C}{\mathbb{C}}

\newcommand{\Z}{\mathbb{Z}}

\newcommand{\pr}{\on{pr}}

\newcommand\lie[1]{\mathfrak{#1}}

\newcommand{\h}{\lie{h}}
\newcommand{\g}{\lie{g}}

\renewcommand{\t}{\lie{t}}
\newcommand{\Alc}{\mathfrak{A}}

\newcommand{\on}{\operatorname}
 
\newcommand{\Aut}{ \on{Aut} }
 
\newcommand{\Ad}{ \on{Ad} }

\newcommand{\Hom}{ \on{Hom}} 
\renewcommand{\ker}{ \on{ker}} 
 
\newcommand{\SU}{ \on{SU}}

\newcommand{\Mult}{ \on{Mult}}

\newcommand\dirac{/\kern-1.2ex\partial} 
\newcommand\qu{/\kern-.7ex/} 
\newcommand{\Waff}{W_{\on{aff}}} 

\newcommand{\lra}{\longrightarrow}
\newcommand{\hra}{\hookrightarrow}

\renewcommand{\d}{{\mbox{d}}}
\newcommand{\ol}{\overline}

\newcommand\Phinv{\Phi^{-1}}

\newcommand\Sig{\Sigma}
\newcommand\sig{\sigma}

\newcommand\Om{\Omega}
\newcommand\om{\omega}

\newcommand{\f}{\frac}

\renewcommand{\l}{\langle}
\renewcommand{\r}{\rangle}
\newcommand{\hh}{{\textstyle \f{1}{2}}}

\newcommand\beqn{\begin{equation}}      
\newcommand\eeqn{\end{equation}}      
\newcommand{\ca}{\mathcal}

\newcommand{\wt}{\widetilde}
\newcommand{\mf}{\mathfrak}
\newcommand{\beq}{\begin{eqnarray*}}
\newcommand{\eeq}{\end{eqnarray*}}

\renewcommand{\subset}{\subseteq}
\newcommand{\sz}{\mathsf{s}}
\newcommand{\tz}{\mathsf{t}}

\begin{document}

\title[]{Convexity for twisted conjugation}

\vskip.2in
\author{E. Meinrenken}
\address{University of Toronto, Department of Mathematics,
40 St George Street, Toronto, Ontario M4S2E4, Canada }
\email{mein@math.toronto.edu}
\date{\today}
\begin{abstract} 
Let $G$ be a compact, simply connected Lie group.  If $\Co_1,\Co_2$ are two $G$-conjugacy classes, then the set of elements in $G$ that can be written as products $g=g_1g_2$ of elements $g_i\in \Co_i$ is invariant under conjugation, and its image under the quotient map $G\to G/\Ad(G)=\Alc$ is a convex polytope. In this note, we will prove an analogous statement for \emph{twisted conjugations} relative to group automorphisms. The result will be obtained as a special case of a convexity theorem for group-valued moment maps which are equivariant with respect to the twisted conjugation action. 
\end{abstract}
\maketitle 
\section{Introduction}
Let $G$ be a compact connected Lie group, with maximal torus $T$, and let $\g,\t$ be their Lie algebras. Fix a positive Weyl chamber $\t_+\subset \t$, and denote by $p\colon \g\to \t_+$ the quotient map, with fiber $p^{-1}(\xi)=\O_\xi$ the adjoint orbit of $\xi$. For any $r>1$, the set
\begin{equation}\label{eq:sum}
 \{(\xi_1,\ldots,\xi_r)\in \t_+\times \cdots \times \t_+\big|\ \exists \zeta_i\in \O_{\xi_i}\colon \ 
\zeta_1+\ldots +\zeta_r=0\}
\end{equation}
is a convex polyhedral known as the \emph{Horn cone}. Fixing $\xi_1,\ldots,\xi_{r-1}$, the Horn cone describes the set of adjoint orbits contained in the sum of adjoint orbits 
$\O_{\xi_1}+\ldots+\O_{\xi_{r-1}}$. For the case of $G=\on{U}(N)$, 
the projection $p(\zeta)$ signifies the set of eigenvalues of a Hermitian matrix $\zeta$, 
hence the Horn cone thus describes the possible eigenvalues of sums of Hermitian matrices with prescribed eigenvalues. The defining inequalities for the $\mf{u}(n)$-Horn cone were obtained by Klyachko \cite{kl:st}, who gave a description in terms of the Schubert calculus of the Grassmannian. This  was extended to arbitrary compact groups by Berenstein-Sjamaar \cite{be:coa}. See  Ressayre \cite{res:ge} and Vergne-Walter \cite{ver:in} for recent developments. 

Suppose in addition that $G$ is simply connected. Let $\Alc\subset \t_+$ be the Weyl alcove. 
Then $\Alc$ labels the set of conjugacy classes in $G$, in the sense that there is a quotient map 
$q\colon G\to \Alc$, with fiber $q^{-1}(\xi)=\Co_\xi$ the conjugacy class of $\exp(\xi)$. As observed in 
Meinrenken-Woodward \cite[Corollary 4.13]{me:lo}, the set 
\begin{equation}\label{eq:product1}
 \{(\xi_1,\ldots,\xi_r)\in \Alc\times \cdots \times \Alc\big|\ \exists g_i\in \Co_{\xi_i}\colon\  g_1\ldots g_r=e\}
 \end{equation}
is a convex polytope. Put differently, this polytope describes the conjugacy classes arising in
products of a collection of prescribed conjugacy classes. In the case of $G=\SU(n)$, it describes the possible eigenvalues of \emph{products} of special unitary matrices with prescribed eigenvalues; these eigenvalue inequalities were determined, in terms of quantum Schubert calculus on flag manifolds,  by Agnihotri-Woodward \cite{ag:ei} and Belkale \cite{bel:loc}. (See also  Belkale-Kumar \cite{bel:mu}.) This was extended to general $G$ by Teleman-Woodward \cite{te:pa}. 

In this note we will show that there are similar polytopes for \emph{twisted} conjugations. Recall that the twisted conjugation action relative to a group automorphism 
$\kappa\in \Aut(G)$ is the action
\begin{equation}\label{eq:twistaction}
 \Ad_g^{(\kappa)}(a)=g\,a\,\kappa(g^{-1}).\end{equation}
As we will explain, 
it suffices to consider automorphisms $\kappa$ defined by Dynkin diagram automorphisms. These automorphisms preserve $\t$, with fixed point set $\t^\kappa$, and there is a convex polytope (\emph{alcove}) $\Alc^{(\kappa)}\subset \t^\kappa$ with a 
quotient map $q^{(\kappa)}\colon G\to \Alc^{(\kappa)}$ whose fiber $(q^{(\kappa)})^{-1}(\xi)=\Co_\xi^{(\kappa)}$ is the $\kappa$-twisted conjugacy class of $\exp(\xi)$. 
\begin{theorem}\label{th:convexity}
Let $\kappa_1,\ldots, \kappa_r$ be diagram automorphisms with 
$\kappa_r\circ \cdots \circ \kappa_1=1$. Then the set 
\begin{equation}\label{eq:product2}
 \{(\xi_1,\ldots,\xi_r)\in \Alc^{(\kappa_1)}\times \cdots\times \Alc^{(\kappa_r)}\big|\ \exists g_i\in \Co_{\xi_i}^{(\kappa_i)}\colon \ \ g_1\cdot\ldots\cdot g_r=e\}
 \end{equation}
 is a convex polytope.
\end{theorem}
It would be interesting to obtain an explicit description of the defining inequalities of the polytopes \eqref{eq:product2}. (In Section \ref{sec:example}, we will work out the case of $G=\SU(3)$ and $r=3$ by direct computation.) Note that these 
polytopes \eqref{eq:product2} arise if one considers products of conjugacy classes of \emph{disconnected} compact Lie groups $K$; indeed each conjugacy class of $K$ is a finite union of twisted conjugacy classes of the identity component $G=K_0$. 

We will obtain Theorem \ref{th:convexity} as a special case of a convexity result for 
group-valued moment maps that are equivariant under \emph{twisted conjugation}. Examples of such spaces are the twisted conjugacy classes, or components of moduli spaces of flat connections for disconnected groups on surfaces with boundary. We have (cf.~ 
Theorem \ref{th:convex}):
\begin{theorem}
Let $(M,\om,\Phi)$ be a compact, connected  q-Hamiltonian $G$-space with a $\kappa$-twisted equivariant moment map 
$\Phi\colon M\to G$. Then 
the fibers of the moment map are connected, and the image 
\[ \Delta(M):=q^{(\kappa)}(\Phi(M))\subset \Alc^{(\kappa)}\]
is a convex polytope.  
\end{theorem}

In a very recent paper, Boalch and Yamakawa \cite{boa:twi} independently considered twisted group-valued moment maps in the context of twisted wild character varieties, generalizing earlier results of Boalch \cite{boa:qu,boa:geo}. In particular, their work has a discussion of twisted moduli spaces, similar to Section \ref{subsec:twistmod}. I also learned about a forthcoming article by Alex Takeda, using twisted group-valued moment maps in the setting of shifted symplectic geometry. 

\vskip.3in
\noindent{\bf Acknowledgments.} I am grateful to Tom Baird for discussions on twisted conjugations and twisted moduli spaces, and to the referee for requesting an explicit example. 

\section{Twisted conjugation}
We begin by reviewing some background material on twisted conjugation actions. 
References include Baird \cite{bai:cl}, Kac \cite{kac:inf}, Mohrdieck \cite{moh:the}, Mohrdieck-Wendt \cite{moh:int}, and Springer \cite{spr:twi}. 

\subsection{Twisted conjugation}
Let $\on{Aut}(G)$ be the group of automorphisms of a Lie group $G$, and 
$\on{Inn}(G)\cong G/Z(G)$ the normal subgroup of inner automorphisms $\Ad_a,\ a\in G$. The quotient group is denoted  $\on{Out}(G)=\on{Aut}(G)/\on{Inn}(G)$. 
For $\kappa\in \on{Aut}(G)$, define the  $\kappa$-twisted conjugation action as 
\[ \Ad_g^{(\kappa)}(h)=gh\kappa(g^{-1}).\] 
Its orbits $\Co\subset G$ are called the $\kappa$-\emph{twisted conjugacy classes}. In terms of the semi-direct product $G\rtimes \on{Aut}(G)$, the twisted conjugation action can be regarded
as an ordinary conjugation, 
\[ (g,1)(h,\kappa)(g^{-1},1)=(gh\kappa(g^{-1}),\kappa).\]
For this reason, we will sometimes use the notation $G\kappa$ for the space  $G$, regarded as a $G$-space under $\kappa$-twisted conjugation. For later reference, we note that if 
$\kappa_1,\kappa_2$ are two automorphisms, then 
\begin{equation}\label{eq:later}
 \Ad_g^{(\kappa_2\kappa_1)}(h_1h_2)=\Ad_g^{(\kappa_1)}(h_1) \Ad_{\kappa_1(g)}^{(\kappa_2)}(h_2)\end{equation}
for all $g,h_1,h_2\in G$. 

The differential of $\kappa\in \on{Aut}(G)$ at the group unit is an automorphism of the Lie algebra $\g$, still denoted by $\kappa$. 
The generating vector fields for the $\kappa$-twisted conjugation action are $\xi_G=\kappa(\xi)^L-\xi^R$ for $\xi\in\g$. In terms of \emph{right} trivialization of the tangent bundle, we have $\xi_G(h)=(\Ad_h\circ\,\kappa-I)\xi$. Hence, the Lie algebra of the stabilizer of $h\in G$ is 
\begin{equation}\label{eq:stab} \g_h=\ker(\Ad_h\circ\,\kappa-I),\end{equation}
while the tangent space to the twisted conjugacy class $\Co=\Ad_G^{(\kappa)}(h)$ is 
\begin{equation}\label{eq:tangent}T_h\Co=\on{ran}(\Ad_h\circ\,\kappa-I),\end{equation}
in \emph{right} trivialization $T_hG=\g$.

Suppose $\kappa'=\Ad_a\circ\, \kappa$ for some $a\in G$. Then the corresponding twisted conjugations are related by right multiplication $r_a\colon G\to G$: 
\[ r_a\circ \Ad_g^{(\kappa')}=\Ad_g^{(\kappa)}\circ\, r_a.\] 
That is, $g\mapsto ga^{-1}$ defines a $G$-map $G\kappa\to G\kappa'$. 
In particular, if $\Co$ is a $\kappa$-twisted conjugacy class then $\Co'=r_{a^{-1}}(\Co)$ is a 
$\kappa'$-twisted conjugacy class. 
\begin{example}\label{ex:conjugacy}
Suppose $\kappa_1,\ldots,\kappa_r\in \Aut(G)$, and let $\Co_i$ be $\kappa_i$-twisted 
conjugacy classes. Then the subset 
\[ \Co_1\cdots \Co_r:=\{h_1\cdots h_r|\ h_i\in \Co_i\}\subset G\] 
is invariant under  $\kappa:=\kappa_r\cdots \kappa_1$-twisted conjugation. This follows by induction from 
\eqref{eq:later}.
Let $\kappa_i'=\Ad_{a_i}\circ\,\kappa_i$ for some $a_i\in G$, and put $\Co'_i=r_{a_i^{-1}}(\Co_i)$ and 
$\kappa'=\kappa_r'\cdots\, \kappa_1'$. Then the problem of finding 
$h_i\in \Co_i$ with product $h_1\cdots h_r$ in a prescribed $\kappa$-twisted conjugacy class $\Co$ is equivalent to 
a similar problem for the $\Co_i'$. 

To see this, let $u_1,\ldots,u_{r+1}$ be inductively defined as 
$ u_{i+1}=a_i \kappa_i(u_i) $ with $u_1=e$, and put $a=u_{r+1}$. Then $\kappa'=\Ad_a \circ\,\kappa$, hence $\Co'=r_{a^{-1}}(\Co)$ is a $\kappa'$-twisted 
conjugacy class. A straightforward calculation shows that if $h_i\in \Co_i$ satisfy $h:=h_1\cdots h_r\in \Co$, then the elements 
\[ h_i'=\Ad_{u_i}^{(\kappa_i)}(h_i) \ a_i^{-1}\in \Co_i'\]
have product $h'=h a^{-1}\in \Co'$. 
\end{example}

\subsection{Diagram automorphisms}
Let $G$ be a  compact and simply connected Lie group, with maximal torus $T$ and Weyl group 
$W=N_G(T)/T$. Fix a positive Weyl chamber $\t_+\subset \t$, with corresponding alcove $\Alc\subset \t_+$. The walls of the Weyl chamber are defined by the simple roots $\alpha_1,\ldots,\alpha_l\in \t^*$.  Let $\alpha_i^\vee\in\t$ be the simple coroots, and let $e_i,\,  f_i\in\g^\C$ be the Chevalley generators, for $i=1,\ldots,\,l$. 

Consider an automorphisms of the Dynkin diagram, given by a bijection $i\mapsto i'$ of its set of vertices preserving 
all Cartan integers: $\l\alpha_i,\alpha_j^\vee\r=\l\alpha_{i'},\alpha_{j'}^\vee\r$. 
Any diagram automorphism defines a unique Lie algebra automorphism $\kappa\in\on{Aut}(\g^\C)$ such that 
$\kappa(e_i)= e_{i'},\ 
\kappa(f_i)=f_{i'}$.  This automorphism preserves the real Lie algebra $\g\subset \g^\C$, and exponentiates to the Lie group $G$. We will refer to the resulting  $\kappa\in \Aut(G)$ itself as a diagram automorphism. Every element of $\on{Out}(G)=\on{Aut}(G)/\on{Inn}(G)$ is represented by a unique diagram automorphism, and the resulting splitting $\on{Out}(G)\hra \on{Aut}(G)$ identifies
\[ \on{Aut}(G)=\on{Inn}(G)\rtimes \on{Out}(G).\]
That is, any automorphism of $G$ can be written as $\kappa'=\Ad_a\circ\,\kappa$ with $a\in G$ and $\kappa\in\on{Out}(G)$. 
To understand the orbit structure of $\kappa$-twisted conjugation actions, it hence suffices to consider the case that $\kappa\in \on{Aut}(G)$ is a diagram automorphism. In particular, $\kappa$ preserves $T$, with fixed point set $T^\kappa\subset G^\kappa$. 
Let $\t^\kappa,\ \t_\kappa$ be the kernel and range of $\kappa|_\t-I\colon \t\to \t$. Then  
$\t^\kappa$ is the Lie algebra of $T^\kappa$, and $\t_\kappa=(\t^\kappa)^\perp$ is the orthogonal space in $\t$ (relative to a $W$-invariant metric).  
Put $T_\kappa=\exp(\t_\kappa)$. Then $T=T^\kappa\ T_\kappa$, with finite intersection
\[ T^\kappa\cap T_\kappa.\]
Let $W^\kappa\subset W$ the subgroup of elements $w$ whose action on $\t$ commutes with $\kappa$. For $a\in G$, denote by $G_a$ the stabilizer under the $\kappa$-twisted adjoint action. For Propositions \ref{prop:p1} and \ref{prop:p2} below, see \cite{moh:int},  \cite{spr:twi}, and references therein.

\begin{proposition}\label{prop:p1}
Let $\kappa\in \Aut(G)$ be a diagram automorphism. Then: 
\begin{enumerate}
\item The group $G^\kappa$ contains $T^\kappa$ as a maximal torus, with Weyl group $W^\kappa$. The intersection $\t^\kappa_+=\t^\kappa\cap \t_+$ is a positive Weyl chamber for $G^\kappa$. 
\item Every $\kappa$-twisted conjugacy class $\Co\subset G$ intersects the torus $T^\kappa$ in an orbit of the finite group 
$(T^\kappa\cap T_\kappa)\rtimes W^\kappa$. Here $T^\kappa\cap T_\kappa$ acts by multiplication on $T^\kappa$.
\item For all $a\in T^\kappa$, the stabilizer group $G_a$ under the twisted conjugation action contains $T^\kappa$ as a maximal torus.   
\end{enumerate}
\end{proposition}
 Let $\Lambda=\exp_T^{-1}(e)\subset \t$ be the integral lattice of $T$. Since $G$ is simply connected, it coincides with the coroot lattice of $(G,T)$.  The fixed point set $\Lambda^\kappa\subset \t^\kappa$ is the integral lattice of $T^\kappa$. It is contained in the lattice, 
\[ \Lambda^{(\kappa)}=\exp_{T^\kappa}^{-1}(T^\kappa\cap T_\kappa).\]
%

%
\begin{proposition}\label{prop:p2}
There is a unique closed convex polytope $\Alc^{(\kappa)}\subseteq \t^\kappa_+$,  
containing the origin, such that $G_{\exp\xi}=T^\kappa$ for elements $\xi\in \on{int}(\Alc^{(\kappa)})$, and such that the map 
\[ \Alc^{(\kappa)}\xrightarrow{\exp} G\lra G/\Ad_G^{(\kappa)}\]
is a bijection. 
Furthermore, 
\begin{enumerate}
\item 
The cone over $\Alc^{(\kappa)}$ is $\t_+^\kappa$. 
\item 
For each open face $\sigma\subseteq \Alc^{(\kappa)}$, the stabilizer group $G_\sigma:=G_{\exp\xi}$ of elements $\xi\in \sigma$ does not depend on $\xi$, 
The stabilizer groups satisfy $G_\sigma\supseteq G_\tau$ for $\sigma\subseteq\ol{\tau}$. 
\item 
The group $\Waff^{(\kappa)}=\Lambda^{(\kappa)}\rtimes W^\kappa$ is an affine reflection group, generated by reflections across the facets of $\Alc^{(\kappa)}$, and having $\Alc^{(\kappa)}$ as a  fundamental domain. 
\end{enumerate}
\end{proposition}

Clearly, if $\kappa=1$ then $\Alc^{(\kappa)}$ is just the usual Weyl alcove, parametrizing the set of (untwisted) conjugacy classes in $G$. Note that in general, $\Lambda^{(\kappa)},\ \Alc^{(\kappa)}$ are different from the coroot lattice and alcove of $(\g^\kappa,\t^\kappa)$. 
The group $\Waff^{(\kappa)}$ is the Weyl group of the twisted affine Lie algebra defined by $\kappa$, see Kac \cite{kac:inf}. 

\subsection{Slices}\label{subsec:slices}
The conjugation action of $G$ on itself has distinguished slices, labeled by the faces of the alcove. We will generalize this fact to twisted conjugation actions. 
\begin{lemma}\label{lem:lemma}
Let $\kappa\in \on{Aut}(G)$, where $G$ is compact. Let $\Co\subset G$ be the $\kappa$-twisted conjugacy class of an element $a\in G^\kappa$. 
Then 
\begin{equation}\label{eq:tanspace}
T_aG=T_a(G_a)\oplus T_a\Co.
\end{equation} 
\end{lemma}
\begin{proof}
Pick an $\on{Aut}(G)$-invariant inner product on $\g$, defining a bi-invariant Riemannian metric on the Lie group $G$ which is also invariant under $\kappa$.  Since $\kappa(a)=a$, we have $a\in G_a$, and we obtain $T_aG_a=\g_a$ in right trivialization. On the other hand,  
by \eqref{eq:stab}  and \eqref{eq:tangent} we have
$T_a\Co=\on{ran}(\Ad_a\circ\, \kappa-I)=\g_a^\perp$ in right trivialization. Since the two spaces are orthogonal, the Lemma follows. 
\end{proof}
Using again that $\kappa(a)=a$, the twisted conjugation action of $a$ on $G$ restricts to the usual conjugation action on $G_a$. In particular, $G_a$ is a $\Ad^{(\kappa)}_a$-invariant 
submanifold of $G$. The Lemma shows that any sufficiently small invariant open neighborhood of $a$ in $G_a$ is a 
slice for the twisted conjugation action.

If $G$ is also simply connected, and $\kappa$ is a diagram automorphism, there is a specific `largest' slice, as follows. 
For any face $\sigma\subset \Alc^{(\kappa)}$, let $\Alc^{(\kappa)}_\sigma$ be the relatively open subset 
of $\Alc^{(\kappa)}$ given as the union of faces $\tau\subset \Alc^{(\kappa)}$ such that $\sigma\subseteq\ol{\tau}$. Put 
\begin{equation}\label{eq:usigma}
 U_\sigma=\Ad_{G_\sigma}^{(\kappa)}\exp(\Alc_\sigma^{(\kappa)}),\end{equation}
a subset of $G_\sigma\subset G$. 
\begin{proposition}\label{prop:slice}
The subset $U_\sigma\subset G_\sigma$ is open, and invariant under the twisted conjugation action of $G_\sigma$. The map 
\begin{equation}\label{eq:actionmap}
G\times_{G_\sigma} U_\sigma\to G,\ [(g,a)]\mapsto \Ad_g^{(\kappa)}a
\end{equation}
is an embedding as an open subset of $G$. That is, $U_\sigma$ is a slice for the twisted conjugation action. 
\end{proposition}
\begin{proof}
Pick $\zeta\in\sigma$, and put $c=\exp\zeta$ so that $G_c=G_\sigma$. 
For all $\xi\in \t^\kappa$ and $g\in G_\sigma$, 
\begin{equation}\label{eq:c}
 \Ad_g^{(\kappa)}\exp(\xi)=
\Ad_g^{(\kappa)}(\exp(\xi-\zeta)c)=
\Ad_g(\exp(\xi-\zeta))\ c.\end{equation}
%
%
It follows that $U_\sigma=U_\sigma'\,c$ where 
\[ U_\sigma'=\Ad_{G_\sigma}\exp\big( \Alc_\sigma^{(\kappa)}-\zeta\big).\]
Equation \eqref{eq:c} also shows that for $\xi\in\Alc_\sigma^{(\kappa)}\subset \t^\kappa$, the stabilizer of $\exp(\xi)$ under the twisted conjugation action of $G$ (which lies in $G_\sigma$, by definition of $\Alc_\sigma^{(\kappa)}$) equals the stabilizer of $\exp(\xi-\zeta)$ under the usual conjugation action of $G_\sigma$. 
Consequently, $\Alc_\sigma^{(\kappa)}-\zeta$ is a relatively open subset of an alcove of $(G_\sigma,T^\kappa)$. It follows that $U_\sigma'$ is open in $G_\sigma$, and hence $U_\sigma$ is open in $G_\sigma$. 

We next show that the map \eqref{eq:actionmap} is injective. Thus suppose $\Ad_g^{(\kappa)}a=\Ad_{g'}^{(\kappa)}a'$, where $a,a'\in U_\sigma$ and $g,g'\in G$. Since $a,a'$ are in the 
same twisted conjugacy class, there is a unique $\xi\in \Alc_\sigma$ and elements $h,h'\in G_\sigma$ such that 
\[ a=\Ad_h^{(\kappa)} \exp(\xi),\ \ \ \ a'=\Ad_{h'}^{(\kappa)} \exp(\xi).\]
We thus obtain
\[ \Ad^{(\kappa)}_{gh} \exp(\xi)=\Ad^{(\kappa)}_{g'h'} \exp(\xi),\]
which implies $ghk=g'h'$ for some $k\in G_{\exp\xi}\subset G_\sigma$. 
Setting $u=h' k^{-1} h^{-1}\in G_\sigma$, we obtain 
$g'=g u^{-1}$, while $a'=\Ad_u^{(\kappa)} a$. That is, $[(g,a)]=[(g',a')]$. 

To complete the proof, it suffices to show that \eqref{eq:actionmap} has surjective differential. By equivariance, it is enough to verify this at elements $[(e,a)]$ with  $a\in \exp(\Alc_\sig^{(\kappa)})\subset T^\kappa$. The range of the differential of \eqref{eq:actionmap} at such a point contains $T_aG_\sigma+T_a\Co$. Since $G_a\subset G_\sigma$, hence 
$T_a G_a\subset T_aG_\sigma$,  Lemma \ref{lem:lemma} shows that this is all of $T_aG$. 
\end{proof}

\section{q-Hamiltonian spaces}
Let $G$ be a Lie group, with an invariant inner product $\cdot $ on its Lie algebra $\g$, and let $\eta\in \Om^3(G)$ be the bi-invariant closed 3-form 
\[ \eta=\f{1}{12}\theta^L\cdot [\theta^L,\theta^L]
=\f{1}{12}\theta^R\cdot [\theta^R,\theta^R]
\]
where $\theta^L,\theta^R\in \Om^1(G,\g)$ are the left, right invariant Maurer-Cartan forms. 
Suppose $\kappa\in \Aut(G)$ is an automorphism. It will be convenient to denote the group $G$, viewed as a $G$-manifold under $\kappa$-twisted conjugation, by $G\kappa$. 

\subsection{$G\kappa$-valued moment maps}
A \emph{q-Hamiltonian $G$-space with $G\kappa$-valued moment map} is a $G$-manifold $M$, together with an $G$-invariant 2-form $\om$ 
and a $G$-equivariant smooth map $\Phi\colon M\to G\kappa$. These are required to satisfy the following axioms: 
\begin{enumerate}
\item $\d\om=-\Phi^*\eta$, 
\item $\iota(\xi_M)\om=-\hh\Phi^*(\kappa(\xi)\cdot \theta^L+\xi\cdot \theta^R)$, 
\item $\ker(\om)\cap \ker(T\Phi)=0$. 
\end{enumerate}
These axioms generalize the $G$-valued moment maps from \cite{al:mom}.  
In terms of equivariant de Rham forms, the first two properties may be combined into a single condition 
$\d_G\omega=-\Phi^*\eta_G^{(\kappa)}$, where 
\[ \eta_G^{(\kappa)}(\xi)=\eta-\hh(\kappa(\xi)\cdot \theta^L+\xi\cdot \theta^R).\]
is a closed equivariant 3-form on $G\kappa$. 

\begin{example}[Twisted conjugacy classes]
$\kappa$-twisted conjugacy classes $\Co\subset G$  are q-Hamiltonian $G$-spaces, with the
$G\kappa$-valued moment map given as the inclusion.  The 2-form is uniquely determined by the moment map condition (b), and is given by 
\[ \om(\xi_\Co,\tau_\Co)=\hh((\Ad_\phi\circ\,\kappa)-(\Ad_\phi\circ\,\kappa)^{-1})\xi\cdot\tau.\]
Note that the twisted conjugacy classes can be odd-dimensional. For example, in the case of $G=\SU(3)$ with $\kappa$ given by complex conjugation, the generic stabilizer under twisted conjugation is a circle, and hence the generic twisted conjugacy classes are 7-dimensional. 
\end{example}

\begin{example}[Twisted moduli spaces]
These are associated to any compact oriented surface with boundary, with marked points on the boundary components, with a prescribed homomorphism from the fundamental groupoid into $\on{Aut}(G)$. 
This will be discussed in Section \ref{subsec:twistmod} below. 
\end{example}

\begin{example}\label{ex:fusion}
Further examples are created by \emph{fusion}:
Suppose $M_i$ for $i=1,2$ are two  q-Hamiltonian $G$-spaces with $G\kappa_i$-valued moment map. Then 
$M_1\times M_2$ with the new $G$-action $g.(m_1,m_2)=(g.m_1,\kappa_1(g).m_2)$ and the 2-form  
\[ \om=\om_1+\om_2+\hh \Phi_1^*\theta^L\cdot \Phi_2^*\theta^R\]
becomes a q-Hamiltonian $G$-space with $G\,\kappa_2\kappa_1$-valued 
 moment map $\Phi_1\Phi_2$. 
 %
Properties (a) and (b) may be verified directly; for property (c) it is best to use the Dirac-geometric approach as in Remark \ref{rem:dirac}.
 
For example, if $\Co$ is a $\kappa$-twisted conjugacy class in $G$, and $M$ is a q-Hamiltonian $G$-space with (non-twisted) $G$-valued moment map, then the fusion product $M\times \Co$ is a q-Hamiltonian $G$-space with $G\kappa$-valued moment map. 
Also, if $\Co_i\subset G$ are $\kappa_i$-twisted conjugacy classes, for $i=1,\ldots,r$,  then their fusion product 
$\Co_1\times\cdots \times\Co_r$ is a q-Hamiltonian space with $G\kappa$-valued moment map, where $\kappa=\kappa_r\cdots \kappa_1$. See example \ref{ex:conjugacy}.
\end{example}

\begin{remark}
Let $L^{(\kappa)}G$ be the twisted loop group, consisting of paths $g\colon \R\to G$ with the property that $g(t+1)=\kappa(g(t))$ for all $t$. There is a notion of Hamiltonian  $L^{(\kappa)}G$-space
generalizing that of a Hamiltonian $LG$-space \cite{me:lo}, and by the same proof as for $\kappa=1$ \cite{al:mom} one 
sees that there is a 1-1 correspondence between Hamiltonian  $L^{(\kappa)}G$-spaces with proper moment maps and $\kappa$-twisted q-Hamiltonian $G$-spaces.  
\end{remark}

\begin{remark} \label{rem:dirac}
The definition of $G\kappa$-valued moment maps has a Dirac-geometric interpretation, similar to 
\cite{bur:di} and \cite{al:pur}. Using the notation from \cite{al:pur}, let $\mathbb{A}=\mathbb{T} G_\eta$ be the  standard Courant algebroid 
over $G$, with the Courant bracket twisted by the closed 3-form $\eta$. 
It has a canonical trivialization $\mathbb{A}=G\times (\ol{\g}\oplus \g)$, where $\ol{\g}$ stands for $\g$ with the opposite metric. 
Any Lagrangian Lie subalgebra $\mf{s}\subset \ol{\g}\oplus \g$ defines a Dirac structure
$E_{\mf{s}}=G\times \mf{s}\subset \mathbb{A}$. Taking $\mf{s}$ to be the diagonal, one obtains the Cartan-Dirac structure $E_\Delta=E$. Taking $\mf{s}=\{(\xi,\kappa(\xi)|\ \xi\in\g\}$ for  $\mf{\kappa}\in\Aut(G)$, one obtains a Dirac structure
$E_{\mf{s}}=E^{(\kappa)}$ generalizing the Cartan-Dirac structure. As a Lie algebroid, it is the action Lie algebroid for the 
$\kappa$-twisted conjugation action. For a q-Hamiltonian space $(M,\om,\Phi)$ with $G\kappa$-valued moment map, the pair  $(\Phi,\omega)$ defines a full morphism of Manin pairs, $(\mathbb{T} M,TM)\da (\mathbb{A},E^{(\kappa)})$. 
Conversely, such a morphism defines a $\g$-action on $M$ for which the underlying map $\Phi\colon M\to G$ is equivariant, and it also defines an invariant 2-form on $M$ satisfying the axioms above.  The Dirac-geometric approach explains many of the properties of $G\kappa$-valued moment maps; for example the fusion construction finds a conceptual explanation in terms of a Dirac morphism 
\[ (\Mult_G,\ \varsigma)\colon (\mathbb{A},E^{(\kappa_1)})\times  (\mathbb{A},E^{(\kappa_2)})
\to (\mathbb{A},E^{(\kappa_2\kappa_1)})\]
with the 2-form $\varsigma=\hh \pr_1^*\theta^L\cdot \pr_2^*\theta^R$. See \cite[Section 4.4]{al:pur}. 
\end{remark}

\subsection{Twisted moduli spaces}\label{subsec:twistmod}
Let $\Sigma=\Sigma_h^r$ be a compact, connected, oriented surface of genus $h$ with $r>0$ boundary components,  and let $\V=\{x_1,\ldots,x_r\}$ be a collection of  base points on the boundary components, 
$x_i\in  (\partial\Sigma)_i\cong S^1$. Let 
\[ \pi_1(\Sigma,\V)\rightrightarrows\V\] 
denote the fundamental groupoid, consisting of homotopy classes $\lambda$ of paths for which both the initial point 
$\sz(\lambda)$ and the end point $\tz(\lambda)$ are in $\ca{V}$. Suppose we are given a groupoid homomorphism (`twist') 
\[ \kappa\in \Hom\big(\pi_1(\Sigma,\V),\on{Aut}(G)\big).\]
Such a $\kappa$ may be obtained by assigning elements of $\on{Aut}(G)$ to  a system of free generators of the fundamental groupoid, and extending by the homomorphism property. Let 
\begin{equation} \label{eq:homg}M=\on{Hom}_\kappa\big(\pi_1(\Sigma,\V),G\big) \end{equation}
be the space of $\kappa$-twisted homomorphisms, consisting of maps 
$\lambda\mapsto \phi_\lambda$ such that 
\[ \phi_{\lambda_1\circ \lambda_2}=\phi_{\lambda_1}\ \kappa_{\lambda_1}(\phi_{\lambda_2})\]
whenever  $\sz(\lambda_1)=\tz(\lambda_2)$. (The space $M$ may be regarded as a certain moduli space of flat connections.) \footnote{Alternatively,  $\on{Hom}_\kappa$ is the lift of $\kappa$ to the space $\wt{M}=\Hom\big(\pi_1(\Sigma,\V),G\rtimes \on{Aut}(G)\big)$.} 
The group $\on{Map}(\V,G)=G\times \cdots \times G$ act on the space \eqref{eq:homg} as 
\[ (g.\phi)_\lambda=g_{\tz(\lambda)}\ \phi_\lambda\ \kappa_\lambda(g_{\sz(\lambda)}^{-1}).\]
Let $\kappa_1,\ldots,\kappa_r\in \on{Aut}(G)$ be the values of  $\kappa$ on the oriented boundary loops $\lambda_1,\ldots,\lambda_r$. Then $M$ is  a  q-Hamiltonian $G^r$-space, with a $G\kappa_1\times \cdots \times G\kappa_r$-valued moment map $\Phi$ given by evaluation on boundary loops. We won't describe the 2-form here, since for the case that $\kappa$ takes values in diagram automorphisms it may be regarded as a component of the moduli space of flat
$G\rtimes \on{Out}(G)$-bundles -- see Section \ref{sec:changing} below. 

\begin{remark}
This construction also gives new examples of \emph{non-twisted} q-Hamiltonian spaces. 
For example, take $\Sigma=\Sigma_1^1$ be the surface of genus $1$ with one boundary component. Its fundamental group(oid) has free generators $\alpha,\beta$, with the boundary loop given as $\alpha\beta\alpha^{-1}\beta^{-1}$. Attach an automorphism $\sigma\in 
\Aut(G)$ to $\beta$, and $1$ to $\alpha$, and extend to a homomorphism $\kappa$ as above. 
Then the corresponding $M$ is $G\times G$, with elements $(a,b)$ corresponding to holonomies along $\alpha,\beta$. The group $G$ acts on $a$ by conjugation and on $b$ by $\kappa$-twisted conjugation. The boundary holonomy is a $G$-valued moment map
 \[ (a,b)\mapsto a b \kappa(a^{-1})\kappa(b^{-1}),\]
a twisted group commutator.
\end{remark}
\begin{remark}
Let $K$ be a disconnected group with identity component $G=K_0$. The space $\Hom(\pi_1(\Sigma,\V),K)$ is a moduli space of flat $K$-bundles over $\Sigma$, with framings at the base points. This space is disconnected, in general. The conjugation action of $K$ on its identity component defines a group homomorphism $K\to \on{Aut}(G)$. Hence, any 
element $x\in \Hom(\pi_1(\Sigma,\V),K)$ determines an element $\kappa\in \Hom(\pi_1(\Sigma,\V),\on{Aut}(G))$, and the connected component containing $x$ is identified with a component of $\Hom_\kappa(\pi_1(\Sigma,\V),G)$.
\end{remark}

The example giving rise to the convex polytope in Theorem \ref{th:convexity} is 
obtained from the $r$-holed sphere $\Sig=\Sig_0^r$. Here $\pi_1(\Sigma,\V)$ 
is freely generated by $\lambda_1,\ldots,\lambda_{r-1}$, represented by oriented
boundary loops based at $x_1,\ldots,x_{r-1}$, together with $\mu_1,\ldots,\mu_{r-1}$ represented by  non-intersecting paths connecting these points to $x_r$. The element $\lambda_r$ represented by the remaining boundary loop satisfies 
\begin{equation}\label{eq:lambdas}
 \prod_{i=1}^{r-1}\mu_i\lambda_i\mu_i^{-1}=\lambda_r^{-1}.
 \end{equation}
Given $\kappa\in \Hom(\pi_1(\Sigma,\V),\on{Aut}(G))$, we denote by $\kappa_i$ the images of the $\lambda_i$'s, and by $\sigma_i$ the images of the $\mu_i$'s. 
Then 
\begin{equation}\label{eq:kappas}
 \prod_{i=1}^{r-1}\sigma_i\kappa_i\sigma_i^{-1}=\kappa_r^{-1}.
 \end{equation}
We find $\Hom_\kappa(\pi_1(\Sigma,\V),G)=G^{2r-2}$, consisting of tuples
$(d_1,\ldots,d_{r-1},a_1,\ldots,a_{r-1})$, where $d_i$ are holonomies attached to the $\lambda_i$, and $a_i$ are attached to the $\mu_i$. The holonomy $d_r$ 
around the $r$-th boundary loop is determined from
\begin{equation}\label{eq:holes}
 \prod_{i=1}^{r-1} (a_i,\sigma_i)(d_i,\kappa_i)(a_i,\sigma_i)^{-1}=(d_r,\kappa_r)^{-1}.
 \end{equation}

\begin{lemma}\label{lem:modspace0r}
Let $\kappa_1,\ldots,\kappa_r$ be holonomies attached to the boundaries of $\Sigma_0^r$, with 
$\kappa_r\kappa_{r-1}\cdots \kappa_1=1$. Then there is an extension to a homomorphism $\kappa\in\Hom(\pi_1(\Sigma,\ca{V}),\on{Aut}(G))$, in such a way that the moment map image of $M=\on{Hom}_\kappa(\pi_1(\Sigma,\ca{V}),\,G)$
consists of all $(d_1,\ldots,d_r)\in G^r$ for which there
exists $(g_1,\ldots,g_r)$ with $g_i\in \Ad_G^{(\kappa_i)}(d_i)$ and 
$\prod_{i=1}^r g_i=e$. 
\end{lemma}
\begin{proof} 
Using the notation above, put 
$\sigma_1=1,\ \sigma_2=\kappa_1^{-1}, \ldots,\ \sigma_{r-1}=\kappa_1^{-1}\cdots \kappa_{r-2}^{-1}$. 
Equation \eqref{eq:kappas} becomes the condition  $\kappa_r\kappa_{r-1}\cdots \kappa_1=1$. Introducing 
\[ a_1'=a_1,\ a_2'=\kappa_1(a_2),\ a_3'=\kappa_2(\kappa_1(a_3)),\ \ldots\]
the equation for the holonomies becomes 
\[ \prod_{i=1}^{r} a_i' d_i \kappa_i((a_i')^{-1}) =e\]
where we put $a_r'=e$. That is $\prod g_i=e$ where 
\[ g_i= a_i' d_i \kappa_i((a_i')^{-1})\in \Ad_G^{(\kappa_i)}(d_i).
\]
The moment map for $M$ is the map taking $(d_1,\ldots,d_{r-1},a_1,\ldots,a_{r-1})$ to 
$(d_1,\ldots,d_r)$, with $d_r$ determined from the condition $\prod_i g_i=e$. 
\end{proof}

\subsection{Basic properties of $G\kappa$-valued moment maps}
The following statement extends a well-known property of moment maps in symplectic geometry. 
\begin{proposition}
Let $(M,\om,\Phi)$ be a q-Hamiltonian $G$-space with $G\kappa$-valued moment map. For all $m\in M$ we have
\[ \ker(T_m\Phi)^\om=T_m(G\cdot m),\ \ \ 
\on{ran}(\Phi^*\theta^R)_m=\g_m^\perp.\]
(For any subspace $V\subset T_mM$, the notation $V^\om$ means the set of all $v\in T_mM$ such that $\omega(v,w)=0$ for all $w\in V$.) 
\end{proposition}
\begin{proof}
In terms of $A=\Ad_{\Phi(m)}\circ\,\kappa$, the moment map condition gives 
\begin{equation}\label{eq:rewrite}
 \iota(\xi_M)\om_m =-\hh((A+I)\xi)\cdot (\Phi^*\theta^R)_m.
\end{equation}
In particular, for  $\xi\in\g_m$, we get that 
\[ \hh ((A+I)\xi )\cdot (\Phi^*\theta^R)_m=0.\]
But  $\g_m\subset \g_{\Phi(m)}=\ker(A-I)$, so $A$  acts as the identity on $\g_m$. Hence we obtain $\xi\cdot (\Phi^*\theta^R)_m=0$, proving $\on{ran}(\Phi^*\theta^R)_m\subset \g_m^\perp$. On the other hand, it is immediate from the moment map condition that $\ker(T_m\Phi)^\om\supseteq T_m(G\cdot m)$. Equality of both inclusions follows by a dimension count: 
\[ \begin{split}
\dim(G\cdot m) &\le \dim(\ker(T_m\Phi)^\om)
\\
& =\dim T_mM-\dim (\ker(T_m\Phi))\\
&=\dim(\on{ran}(\Phi^*\theta^R)_m)\\
& \le \dim \g_m^\perp=\dim(G\cdot m).
\end{split}\]
Here we used $\ker(\om)\cap \ker(T\Phi)=0$ for the first equality sign. 
\end{proof}
\begin{proposition}
The map $\g\to T_mM$given by the infinitesimal action restricts to an isomorphism, 
\[ \ker(\Ad_{\Phi(m)}\circ\,\kappa+I)\xrightarrow{\cong} \ker(\om_m).\]
\end{proposition}
\begin{proof}
Here, the Dirac-geometric viewpoint from Remark \ref{rem:dirac} is convenient. Let $\mathbb{T} G_\eta$ be 
as in that remark. The subspace 
\[ E_1=\{T\Phi(v)+\alpha\in \mathbb{T} G_\eta|\ v\in T_mM,\ \alpha\in T^*_{\Phi(m)}G,\ 
\Phi^*\alpha=\iota(v)\omega_m\}\]
is the `forward image' of $T_mM\subset \mathbb{T} M=TM\oplus T^*M$ under the linear Dirac morphism 
$(T_m\Phi,\omega_m)$; in particular it satisfies $E_1=E_1^\perp$. 
The axioms show that $E_1$ contains the space 
\[ E=\big\{\xi_G +\f{1}{2}\theta^R\cdot(A+I)\xi\ \big|\ \ \xi\in\g\big\}\]
(everything evaluated at $\Phi(m)$). Here $\xi_G$ are the generating vector fields for the $\kappa$-twisted conjugation, 
\[ \xi_G=\kappa(\xi)^L-\xi^R=((A-I)\xi)^R.\]
But it is easily checked that $E=E^\perp$, which together with $E\subset E_1$ implies $E_1=E$. 
In particular, taking $\alpha=0$ in the definition of $E_1$ we see that
\[ (T_m\Phi)(\ker\om_m)=\big\{\xi_G(\Phi(m))\big|\ (A+I)\xi=0\big\}.\]
Since $\ker(\om_m)\cap \ker(T_m\Phi)=0$, the map $T_m\Phi$ is injective on  
$\ker(\om_m)$. Consequently, $\ker(\om_m)=\{\xi_M(m)|\ (A+I)\xi=0\}$. 
\end{proof}

\subsection{Changing $\kappa$ by inner automorphisms}\label{sec:changing}
Suppose $\kappa'=\Ad_a \circ\,\kappa$, and let $(M,\om,\Phi)$ be a  q-Hamiltonian  $G$-space with $G\kappa$-valued moment map. Then the manifold $M$ with the same $G$-action and 2-form, but with a shifted moment map $\Phi'=r_{a^{-1}}\circ \Phi$, is a q-Hamiltonian $G$-space with  $G \kappa'$-valued moment map. 
For this reason, if $G$ is compact and simply connected, it usually suffices to consider the case of diagram automorphism $\kappa\in\on{Out}(G)$. But for $\kappa\in\on{Out}(G)$, the 
q-Hamiltonian $G$-spaces with $G\kappa$-valued moment map are simply 
q-Hamiltonian spaces with moment maps valued in the disconnected group 
$G\rtimes \on{Out}(G)$, whose image is contained in the component $G\times \{\kappa\}$. 
(The only wrinkle is that we only consider the action of the identity component $G$ of this group, but this doesn't affect the theory from \cite{al:mom}.)  In this sense, the examples considered above are not new, at least for $G$ compact and simply connected. For instance, 
in the fusion procedure \ref{ex:fusion}, first apply the automorphism $\kappa_1$ to the 
second space, thus obtaining $(M_2,\om_2,\Phi_2')$ with the new $G$-action 
$m\mapsto \kappa_1(g).m$, and a $G\kappa'$-valued moment map 
$\Phi_2'=\kappa_1^{-1}\circ \Phi_2$, where $\kappa'=\kappa_1^{-1}\kappa_2\kappa_1$. 
Since
\[ (\Phi_1,\kappa_1)(\kappa_1^{-1}\circ\Phi_2,\kappa_2)=(\Phi_1\Phi_2,\kappa_2\circ\,\kappa_1),\]
we recognize the fusion product  \ref{ex:fusion} as a standard fusion product \cite{al:mom} for q-Hamiltonian $G$-spaces with $G\rtimes \on{Out}(G)$-valued moment maps.

\section{Convexity properties}
We now turn to the convexity properties of $G\kappa$-valued moment maps. 
The arguments are mostly straightforward adaptations of those in \cite{me:lo} and \cite{le:co}.
Throughout, we will assume that $G$ is compact and simply connected, and that $\kappa\in \on{Aut}(G)$ is a diagram automorphism. We denote by $\Alc^{(\kappa)}\subset \t^\kappa$ the alcove, and by 
\[ q^{(\kappa)}\colon G\to \Alc^{(\kappa)}\] 
the quotient map, with fibers $(q^{(\kappa)})^{-1}(\xi)$ the $\kappa$-twisted conjugacy classes of $\exp(\xi)$. Recall from \ref{subsec:slices} 
the definition of the slices $U_\sigma$. Let $\kappa_\sigma$ denote the restriction of $\kappa$ to $G_\sigma$. 

\begin{proposition}[Cross-section theorem]
Let $(M,\om,\Phi)$ be a connected q-Hamiltonian $G$-space with $G\kappa$-valued moment map. For any face $\sigma\subset \Alc^{(\kappa)}$, the pre-image $Y_\sigma=\Phinv(U_\sigma)$ is a q-Hamiltonian $G_\sigma\kappa_\sigma$-space, 
with the pullback of $\omega$ as the 2-form and the restriction of $\Phi$ as the moment map.  
\end{proposition}
The proof is parallel to the result for non-twisted q-Hamiltonian spaces, see \cite{al:mom}, which in turn is a version of the 
cross-section theorem for Hamiltonian spaces, due to Guillemin-Sternberg \cite{gu:no} and Marle \cite{ma:mo}. 

Recall that for any connected $G$-manifold $M$, the  principal stratum $M_{\on{prin}}$ 
is the set of all points whose stabilizer is subconjugate to any other stabilizer. 
It is connected, and open and dense in $M$.  
\begin{proposition}
Let $(M,\om,\Phi)$ be a connected q-Hamiltonian $G$-space with $G\kappa$-valued moment map. Then: \begin{enumerate}
\item
The stabilizer $G_m$ of any point $m\in M_{\on{prin}}$ is an ideal in $G_{\Phi(m)}$. 
\item 
All points in $M_{\on{prin}}\cap \Phinv(\exp(\Alc^{(\kappa)}))$ have the same stabilizer $H$. 
\item 
The image $q^{(\kappa)}(\Phi(M_{\on{prin}}))$ is a connected, relatively open subset of 
\[ (x+\h^\perp)\cap \Alc^{(\kappa)},\] 
where $\h$ is the 
Lie algebra of $H$, and $x$ is any point of $q^{(\kappa)}(\Phi(M_{\on{prin}}))$. 
\end{enumerate}
\end{proposition}
\begin{proof}
The parallel statements for ordinary Hamiltonian $G$-spaces are proved in 
\cite[Section 3.3]{le:co}. In particular, if $N$ is a connected Hamiltonian $G$-space, with moment map $\Psi\colon N\to \g^*$, then for each $n\in N_{\on{prin}}$, the stabilizer $G_n$ is an ideal in $G_{\Phi(n)}$,  and the stabilizer $H=G_n$ of points in $N_{\on{prin}}\cap \Psi^{-1}(\t^*_+)$ is independent of $n$. We will use cross-sections $Y_\sigma$ 
to reduce to the Hamiltonian case. As noted in the proof of Proposition \ref{prop:slice}, the automorphism $\kappa_\sigma=\kappa|_{G_\sigma}$ 
is inner, and is given by $\Ad_{a^{-1}}$ for any choice of  $a\in \exp(\sigma)$. 
Hence,   $Y_\sigma$ becomes a q-Hamiltonian $G_\sigma$-space  with (untwisted) $G_\sigma$-valued moment map $r_{a^{-1}}\circ \Phi_\sigma$. Furthermore, this then becomes an ordinary Hamiltonian $G_\sigma$-space with a moment map 
\[ \Phi_{0,\sigma}\colon Y_\sigma\to \g_\sigma\cong \g_\sigma^*,\ \ 
m\mapsto \log(\Phi_\sigma(m) a^{-1}).\]
We conclude that for all 
$m\in (Y_\sigma)_{\on{prin}}=Y_\sigma\cap M_{\on{prin}}$, the stabilizer $G_m$ is an ideal in 
the stabilizer of $\Phi_{0,\sigma}(m)$ under the adjoint action. The 
latter coincides with 
stabilizer of $\Phi(m)=\exp(\Phi_{0,\sigma}(m))a$
under twisted conjugation. Hence $G_m$ is an ideal in $G_{\Phi(m)}$. Since the flow-outs of all the $Y_\sigma$'s under twisted conjugation  cover $M$, this proves (a). 

The map $M_{\on{prin}}\cap \Phinv(\exp(\Alc^{(\kappa)}))\to M_{\on{prin}}/G$ is surjective, and has connected fibers $G_{\Phi(m)}.m=G_{\Phi(m)}/G_m$. Since the target of this map is connected, it follows that $M_{\on{prin}}\cap \Phinv(\exp(\Alc^{(\kappa)}))$ is connected. 
Consider the decomposition of each $Y_\sigma$ into its connected components 
$Y_\sigma^i$. 
Passing to the corresponding Hamiltonian $G_\sigma$-space as above, and using the general results for connected Hamiltonian spaces, we see that all points of $Y_\sigma^i\cap M_{\on{prin}}\cap \Phinv(\exp(\Alc^{(\kappa)}))$ 
have the same stabilizer. Since the union of these sets, over all $\sigma,i$, covers 
$M_{\on{prin}}\cap \Phinv(\exp(\Alc^{(\kappa)}))$, it follows that all points of this intersection 
have the same stabilizer, proving (b). 

Each $q^{(\kappa)}(\Phi(M_{\on{prin}})\cap Y_\sigma^i)$  is a connected, relatively open subset of $(x+\h^\perp)\cap \Alc^{(\kappa)}_\sigma$, for any choice of
$x\in q^{(\kappa)}(\Phi(M_{\on{prin}})\cap Y_\sigma^i)$. (Once again, this follows from the corresponding statement for Hamiltonian spaces, see \cite[Section 3.3]{le:co}.) This implies (c).  
\end{proof}
\begin{theorem}[Principal cross-section]
Let $(M,\omega,\Phi)$ be a connected q-Hamiltonian $G$-space with $G\kappa$-valued moment map. Then there exists a unique open face $\sigma$ of $\Alc^{(\kappa)}$ such that 
\[ q^{(\kappa)}(\Phi(M))\subseteq \ol{q^{(\kappa)}(\Phi(M))\cap \sigma}.\]
(Equality holds if $M$ is compact.) Alternatively, $\sigma$ is characterized as the smallest face such that the corresponding cross-section $Y_\sigma$ satisfies 
$\Phi(Y_\sigma)\subset \exp(\sigma)$. This \emph{principal cross-section} $Y_\sigma$ is a connected q-Hamiltonian $T^\kappa$-space, 
with the 
restriction of  $\Phi$ as the moment map, and 
\[ M=\ol{G\cdot Y_\sigma}.\] 
\end{theorem}
\begin{proof}
Using the notation from the previous proposition, let $\sigma$ be the lowest dimensional face of $\Alc^{(\kappa)}$ whose closure contains $(x+\h^\perp)\cap \Alc^{(\kappa)}$. Since $q^{(\kappa)}(\Phi(M_{\on{prin}}))$ is a relatively open subset of 
$(x+\h^\perp)\cap \Alc^{(\kappa)}$, its intersection with $\sigma$ is non-empty. It follows that $q^{(\kappa)}(\Phi(M))\cap\sigma
=q^{(\kappa)}(\Phi(Y_\sigma))$. That is, $\Phi(Y_\sigma)\subset \exp(\sigma)\subset T^\kappa$,  
so that $Y_\sigma$ may be regarded as a q-Hamiltonian $T^\kappa$-space, for the restriction of the moment map. 

By construction, $G\cdot Y_\sigma=\Phinv((q^{(\kappa)})^{-1}(\sigma))$. 
The difference
\begin{equation}\label{eq:difference}
M_{\on{prin}}-((G\cdot Y_\sigma)\cap M_{\on{prin}})=
 M_{\on{prin}}-\big(\Phinv((q^{(\kappa)})^{-1}(\sigma))\cap M_{\on{prin}}\big) 
\end{equation}
is the union over all 
$\Phinv((q^{(\kappa)})^{-1}(\tau))\cap M_{\on{prin}}$ where $\tau$ ranges over proper faces of $\ol{\sigma}$. 
But those are submanifolds of codimension at least $3$, hence removing them will not disconnect $M_{\on{prin}}$. Thus
$(G\cdot Y_\sigma)\cap M_{\on{prin}}$ is connected, which implies that 
$G\cdot Y_\sigma=G\times_{G_\sigma}Y_\sigma$ is connected, and therefore $Y_\sigma$ is connected. 
\end{proof}
Note that since the principal cross-section $Y_\sigma$ is a q-Hamiltonian $T^\kappa$-space, it is in particular symplectic. 
\begin{theorem}\label{th:convex}
Let $(M,\om,\Phi)$ be a compact, connected  q-Hamiltonian $G$-space with $G\kappa$-valued moment map. Then 
the fibers of the moment map $\Phi$ are connected, and the image 
\[ \Delta(M):=q^{(\kappa)}(\Phi(M))\subset \Alc^{(\kappa)}\]
is a convex polytope.  
\end{theorem}
\begin{proof}
The principal cross-section $Y=Y_\sigma$ is a connected q-Hamiltonian $T^\kappa$-space, with the restriction $\Phi_Y=\Phi|_Y$ as its moment map. We can regard $Y$ as an ordinary Hamiltonian $T^\kappa$-space, with a moment map $\Phi_{Y,0}=q^{(\kappa)}\circ \Phi_Y$ that is proper as a map to $\sig\subset \t^\kappa$.  

Since $\sig$ is convex, \cite[Theorem 4.3]{le:co} shows that $\Phi_{Y,0}$ has connected fibers, and its image is a convex set of the form 
\[ q^{(\kappa)}(\Phi(Y))=\Phi_{Y,0}(Y)=P\cap \sigma,\] where $P$ is some convex polytope in $\ol{\sigma}$. But then $q^{(\kappa)}(\Phi(M))=\ol{q^{(\kappa)}(\Phi(Y))}=P$. Finally, if $x\in q^{(\kappa)}(\Phi(M))$, then the same argument as in 
\cite{le:co} shows that for any open ball $B$ around $x$, the pre-image $\Phinv((q^{(\kappa)})^{-1}(B))$ is connected. 
By a continuity argument \cite[Lemma 5.1]{le:co} this implies that $\Phinv(x)$ is connected.  
\end{proof}

We obtain Theorem \ref{th:convexity} as a special case: 
\begin{proof}[Proof of Theorem \ref{th:convexity}]
Consider again the twisted moduli space for the $r$-holed sphere $\Sig_0^r$, 
corresponding to $\kappa_i\in \on{Out}(G)$ with $\kappa_r\kappa_{r-1}\cdots \kappa_1=1$, as in Lemma \ref{lem:modspace0r}.
We had found that the moment map image consists of all $(d_1,\ldots,d_r)$ for which there exist elements $g_i\in G$ in the $\kappa_i$-twisted conjugacy class of $d_i$, such that $g_1\cdots g_r=e$. Hence, by Theorem \ref{th:convex} the set \eqref{eq:product2} is a convex polytope.
\end{proof} 

\section{An example}\label{sec:example}
We will illustrate Theorem \ref{th:convexity} in a simple setting, were the resulting polytope can be computed by hand. Let $G=A_2\cong \SU(3)$, with its standard maximal torus $T$ consisting of diagonal matrices, and its usual choice of positive roots. We denote by $\alpha,\beta$ the simple roots, and let $\gamma=\alpha+\beta$ be their sum. 
The fundamental alcove $\Alc\subset\t$ is defined by the inequalities 
$\l\alpha,\xi\r\ge 0,\ \l\beta,\xi\r\ge 0,\ \l\gamma,\xi\r\le 1$. 
Let $\kappa\in\on{Aut}(G)$ be the nontrivial diagram automorphism of $G$ given by
$\kappa(\alpha)=\beta$ and $\kappa(\beta)=\alpha$. 
 

The Lie algebra $\t^\kappa$ consists of all $\xi$ such that $\l\alpha,\xi\r=\l\beta,\xi\r$; 
it is thus the line spanned by the coroot $\gamma^\vee$. The alcove $\Alc^{(\kappa)}$ is `half' of the intersection  $\Alc\cap \t^\kappa$, i.e. 
it consists of elements of $\t^\kappa$ with $\l\gamma,\xi\r\in [0,\f{1}{2}]$. 
We thus label the $\kappa$-twisted conjugacy classes by a parameter $s\in [0,\f{1}{2}]$, where $\Co^{(\kappa)}_s$ contains $\exp(\xi_s)$ for a unique $\xi_s\in \Alc^\kappa$ with $\l\gamma,\xi_s\r=s$. 

Consider the setting of Theorem \ref{th:convexity}, with $r=3$. Unless all $\kappa_i=1$, two of the automorphisms $\kappa_1,\kappa_2,\kappa_3$ have to be $\kappa$, and the third is the identity. We may assume $\kappa_1=\kappa_2=\kappa$ and $\kappa_3=1$. 
Hence, 
\[ \Alc^{(\kappa_1)}\times \Alc^{(\kappa_2)}\times \Alc^{(\kappa_3)}
=\big[0,\f{1}{2}\big]\times \big[0,\f{1}{2}\big]\times \Alc.\]
\begin{proposition}
For $G=A_2\cong \SU(3)$ with its non-trivial diagram automorphism $\kappa$, the polytope of  all $(s_1,s_2,\xi)\in \big[0,\f{1}{2}\big]\times \big[0,\f{1}{2}\big]\times \Alc$ such that there 
exists $(g_1,g_2,g_3)\in \Co^{(\kappa)}_{s_1}\times \Co^{(\kappa)}_{s_2}\times \Co_\xi$ with $g_1g_2g_3=e$, is given by the inequalities 
$0\le s_i\le \f{1}{2}$ together with 
\[ |s_1-s_2| \le \l\alpha+\beta,\xi\r\le 1\,\ \ \ \ 
  |s_1-s_2| \le 1-\l\alpha,\xi\r\le 1,\ \ \ \ 
    |s_1-s_2| \le 1-\l\beta,\xi\r\le 1.\]
\end{proposition}
\begin{proof}
The problem of computing this polytope is equivalent to computing the moment polytope of the fusion product $\Co^{(\kappa)}_{s_1}\times \Co^{(\kappa)}_{s_2}$ for any $s_1,s_2$. This fusion product is an untwisted q-Hamiltonian $G$-space, with action 
\[ h\cdot(g_1,g_2)=\big(h\,g_1\,\kappa(h)^{-1},\ \kappa(h)g_2\,h^{-1}\big)\]
and moment map $(g_1,g_2)\mapsto g_1g_2$; its moment polytope is a 2-dimensional convex polytope inside $\Alc$. Observe that the set of $g_1g_2$ with $g_i\in \Co^{(\kappa)}_{s_i}$ is invariant under left-translation by central elements $c\in Z(G)\cong \Z_3$. This follows from 
\[ \Ad^\kappa_{c^{-1}}(g)=c^{-1}g\kappa(c)=c^{-1} g c^2=cg.\] 
Left multiplication of the center on $G$ induces an action on the set of conjugacy classes, 
and the resulting action of $\Z_3$ on the alcove $\Alc$ is by `rotation'. 

Hence, the moment polytope is invariant under `rotations' of the alcove. If $s_1=s_2=0$, this implies that the moment polytope must be all of $\Alc$, since it contains the origin. If at least one of $s_1,s_2$ is non-zero, the moment polytope 
does \emph{not} contain the origin. Using standard results from symplectic geometry, applied to the symplectic cross-section, it is cut out from the alcove by affine half-spaces orthogonal to 1-dimensional stabilizer groups. But the generic stabilizer for the twisted 
conjugation action of $G$ on itself is $T^\kappa$, and all other 1-dimensional stabilizers are $W$-conjugate to $T^\kappa$. (The fixed point set of $T$ is trivial.) Together with the rotational symmetry, it  follows that the moment polytope is cut out from the alcove by inequalities of the form 
$r\le \l\gamma,\xi\r,\ \ r\le 1-\l\alpha,\xi\r,\ \ r\le 1-\l\beta,\xi\r$, for some 
$0<r<\f{1}{2}$. To find $r$, it suffices to determine the fixed point set of $T^\kappa$ on the product of twisted conjugacy classes, and takes its image under the multiplication map. Since the action of $T^\kappa$ is just ordinary conjugation, and since $T^\kappa$ contains regular elements, the fixed point set for each factor is 
\[ \Co^{(\kappa)}_{s_i}\cap T=\exp(\xi_{s_i}+\t_\kappa)\cup \exp(-\xi_{s_i}+\t_\kappa),\]
and the image under multiplication is $\exp(\xi_{\pm s_1\pm s_2}+\t^\kappa)\subset T$. We conclude $r=|s_1-s_2|$. 
\end{proof}

\bibliographystyle{amsplain}
\def\cprime{$'$} \def\polhk#1{\setbox0=\hbox{#1}{\ooalign{\hidewidth
  \lower1.5ex\hbox{`}\hidewidth\crcr\unhbox0}}} \def\cprime{$'$}
  \def\cprime{$'$} \def\cprime{$'$} \def\cprime{$'$}
  \def\polhk#1{\setbox0=\hbox{#1}{\ooalign{\hidewidth
  \lower1.5ex\hbox{`}\hidewidth\crcr\unhbox0}}} \def\cprime{$'$}
  \def\cprime{$'$} \def\cprime{$'$} \def\cprime{$'$} \def\cprime{$'$}
\providecommand{\bysame}{\leavevmode\hbox to3em{\hrulefill}\thinspace}
\providecommand{\MR}{\relax\ifhmode\unskip\space\fi MR }
\providecommand{\MRhref}[2]{%
  \href{http://www.ams.org/mathscinet-getitem?mr=#1}{#2}
}
\providecommand{\href}[2]{#2}

\end{document}